\documentclass[a4paper,10pt]{article}
\usepackage[utf8x]{inputenc}
\usepackage{graphicx,pstricks}
\usepackage{graphics}
\usepackage{epsfig}
\usepackage{caption}
\usepackage{palatino}
\usepackage{enumerate}

\usepackage{amsmath}
\usepackage{amssymb}
\usepackage{amsthm}
\usepackage{alltt}
\usepackage{amscd}
\usepackage[mathscr]{eucal}
\usepackage{mathrsfs}
\usepackage[all]{xy}
\usepackage{verbatim}
\usepackage{moreverb}
\usepackage{graphicx}

\usepackage{amsmath}
\usepackage{amsfonts}
\usepackage{amssymb}
\usepackage{amsthm}
\renewcommand{\caption}[1]{\singlespacing\hangcaption{#1}\normalspacing}

\theoremstyle{definition}
\newtheorem{theorem}{Theorem}[section]
\newtheorem{definition}[theorem]{Definition}
\newtheorem{lemma}[theorem]{Lemma}
\newtheorem{corollary}[theorem]{Corollary}

\newtheorem{proposition}[theorem]{Proposition}
\newtheorem{remark}[theorem]{Remark}

\newtheorem{conjecture}[theorem]{Conjecture}

\newcommand{\RR}{{\mathbb R}}
\newcommand{\CC}{{\mathbb C}}
\newcommand{\NN}{{\mathbb N}}

\newcommand{\thetaA}[1]{{#1}(\theta)}
\newcommand{\thetaB}[1]{{#1}(\bar{\theta})}
\newcommand{\prodd}[2]{\displaystyle \prod_{{#1}=1}^{#2}d_{#1}}
\newcommand{\sumd}[2]{\displaystyle \sum_{{#1}=1}^{#2}d_{#1}}
\newcommand{\prodl}[3]{\displaystyle \prod_{{#1}=1}^{#2}\lambda_{#1}^{#3}}

\newcommand{\sumw}[3]{\sum_{{#1}=0}^{#2}{#3}^{#1}_n}
\newcommand{\sierp}{\text{Sierpi\'{n}ski}}

\title{Counting Spanning Trees on Fractal Graphs}
\author{Jason A. Anema}

\begin{document}

\maketitle

\begin{abstract}
Using the method of spectral decimation and a modified version of Kirchhoff's Matrix-Tree Theorem, a closed form solution to the number of spanning trees on approximating graphs to a fully symmetric self-similar structure on a finitely ramified fractal is given in Theorem (\ref{thm:maintheoremfull}). Examples calculated include the $\sierp$ Gasket, a non-p.c.f. analog of the $\sierp$ Gasket, the Diamond fractal, and the Hexagasket. For each example, the asymptotic complexity constant is found. 

Dropping the fully symmetry assumption, it is shown that the limsup and liminf of the asymptotic complexity constant exist. 

\end{abstract}

\section{Introduction}
 The problem of counting the number of spanning trees in a finite graph dates back more than 150 years. It is one of the oldest and most important graph invariants, and has been actively studied for decades. Kirchhoff's famous Matrix-Tree Theorem \cite{Ki1847}, appearing in 1847, relates properties of electrical networks and the number spanning trees. There are now a large variety of proofs for the Matrix-Tree Theorem, for some examples see \cite{Bo98, Ch82, HP73}. Counting spanning trees is a problem of fundamental interest in mathematics \cite[e.g.]{Bi74, We93, BP93, Ly05, BM08} and physics \cite[e.g.]{Wu77, Zi06, Fo72, Wu82, Dh06}. Its relation to probability theory was explored in \cite{Ly10, Me05}. It has found applications in theoretical chemistry, relating to the enumeration of certain chemical isomers \cite{Br96}, and as a measure of network reliability in the theory of networks \cite{Co87}. 

Recently, S.C. Chang et al. studied the number of spanning trees and the associated asymptotic complexity constants on regular lattices in \cite{CS06,CW06,SW00,Tz00}. These types of problems led them to consider spanning trees on self-similar fractal lattices, as they exhibit scale invariance rather than translation invariance. In \cite{CCY07} S.C. Chang, L.C. Chen, and W.S. Yang calculate the number of spanning trees on the sequence of graph approximations to the $\sierp$ Gasket of dimension two, three and four, as well as for two generalized $\sierp$ Gaskets ($SG_{2,3}(n)$ and $SG_{2,4}(n)$), and conjecture a formula for the number of spanning trees on the $d-dimensional$ $\sierp$ Gasket at stage $n$, for general $d$. Their method of proof uses a decomposition argument to derive multi-dimensional polynomial recursion equations to be solved. Independently, that same year, E. Teufl and S. Wagner \cite{TeWa06} give the number of spanning trees on the  $\sierp$ Gasket of dimension two at stage $n$, using the same argument. In \cite{TeWa07} they expand on this work, contructing graphs by a replacement procedure yielding a sequence of self-similar graphs (this notion of self-similarity is different than in \cite{Ki01}), which include the $\sierp$ graphs. For a variety of enumeration problems, including counting spanning trees, they show that their construction leads to polynomial systems of recurrences and provide methods to solve these recurrences asymptotically. 
Using the same construction technique in \cite{TeWa11}, they give, under the assumptions of $strong$ $symmetry$ (see \cite[ section 2.2]{TeWa11}) and $connectedness$, a closed form equation for the number of spanning trees \cite[ Theorem 4.2]{TeWa11}. This formulation requires calculating the resistance scaling factor and the tree scaling factor (defined in \cite[Theorem 4.1]{TeWa11}). In Section 8.3.1 they show that the $d-dimensional$ $\sierp$ Gasket at stage $n$, satisfies their assumptions and prove the conjecture of \cite{CCY07}. $Strong$ $Symmetry$ is a condition which must be satisfied on each level of construction, whereas the $full$ $symmetry$ condition, that will be assumed in the present work, is only a condition on the first level of construction. The Hexagasket, found in section \ref{section:hexa}, is an 
example of a graph sequence that is fully symmetric, and not strongly symmetric.

 In \cite{MR2450694, MR2451619} B. Steinhurst, A. Teplyaev,  et al., describe the method of spectral decimation 
  for self-similar fully symmetric finitely ramified fractals, which shows how to explicitly calculate the spectrum 
  of the Laplacian on such fractals, generalizing the ideas of \cite{FS92,Sh93}. The central result of the present
  work, Theorem (\ref{thm:maintheoremfull}), relies on their paper to describe how to calculate, in an analytic fashion, the number of spanning trees of 
  the sequence of graph approximations to such fractals. The idea is that the number of spanning trees on finite, connected, loopless, graph
is given by a normalized product of the non-zero eigenvalues of the graph's probabilistic Laplacian. 
The fractal graphs considered here have the advantage that we can calculate those eigenvalues explicitly, as
preiterates of a particular rational function. This enables one to be able to calculate their product explicitly, and 
hence calculate the number of spanning trees explicitly. Section 2 of this work will set up some notation and describe
the class of fractal graphs that will be considered. In Section 3 the main result of this work is presented. Theorem (\ref{thm:maintheoremfull})
enables one to write down a closed formula for the number of spanning trees on the class of fractal graphs considered. 
A nice Corollary of this is the fact that such formulas remain simple. In section 4, it is shown that if
we drop the assumption of full symmetry, then the limsup and liminf of the asymptotic complexity constant exist. This section 
ends with a few related conjectures. Section 5 begins with a well known example, the $\sierp$ Gasket. This 
is done to show how to use Theorem (\ref{thm:maintheoremfull}). The section continues to calculate the 
number of spanning trees, and the asymptotic complexity constant, for the graph approximations to a non-p.c.f.
analog of the $\sierp$ Gasket, the Diamond fractal, and the Hexagasket. The author would like to thank Benjamin Steinhurst, Robert Strichartz, and Alexander Teplyaev for their helpful 
conversations and comments.  
  
\section{Background and Preliminaries}
\subsection{Graph and Probabilistic Graph Laplacians}
Kirchhoff's Matrix-Tree Theorem relates a normalized product of the non-zero eigenvalues of the graph Laplacian to the number 
of spanning trees of a loopless connected graph, since fractal graphs are always connected, and loopless we will make 
this assumption henceforward. However, using the method of spectral decimation one is only able to 
find the eigenvalues of the probabilistic graph Laplacian for a specified class of fractal graphs, so a suitable 
version of Kirchhoff's theorem for probabilistic graph Laplacians must be found. Working in that direction, recall
that for any graph $T$ $=$ $(V,E)$ having $n$ labelled vertices $v_1,v_2,...,v_n$, with vertex set $V$ and edge set $E$.
The graph Laplacian $G$ on $T$ is defined by $G=D-A$, where $D=((d_{ij}))$ is the degree matrix on T with
$d_{ij}=0$ for $i\neq j$ and $d_{ii}=deg(v_i)$, and $A=((a_{ij}))$ is the adjacency matrix on T with $a_{ij}$ is the
number of copies of $\{v_i,v_j\}\in E$. The probabilistic graph Laplacian of $T$ is defined by $P=D^{-1}G$.
Let $I$ be the $n\times n$ identity matrix, $$\chi(G)=|G-xI|=\sum_{i=0}^n c_i^G x^i,$$ and $$\chi(P)=|P-xI|=\sum_{i=0}^{n}c_i^P x^i,$$
be the characteristic polynomials of $G$ and $P$, respectively. Let $S:=\{1,2,...,n-1,n\}$.  
If $\theta\subseteq S$, then let $\bar{\theta}$ denote the complement of $\theta$ in $S$.  For any $n\times n$
 matrix $C$ and any $\theta\subseteq S$, let $C(\theta)$ denote the principal submatrix of $C$ formed by deleting 
all rows and columns not indexed by an element of $\theta$. From \cite{collingsB}, we have that for any $m\times m$ diagonal
matrix $B$, and any $m\times m$ matrix $C$, $$|B+C|=\sum_{\theta\subseteq S} |\thetaB{B}|\cdot |\thetaA{C}|,$$
where the summation is over all subsets $S=\{1,...,m\}$. Using this observation and expanding term by term it 
follows that 

%\begin{lemma}\label{coeffGP}  For any graph $T$ with $n$ vertices, the coefficient of $\chi(G)$ and $\chi(P)$ are given by

\begin{equation}\label{coeffG}
c_{n-i}^G=(-1)^{n-i}\sum_{|\theta|=i}|\thetaA{D}|\cdot |\thetaA{P}|
\end{equation}
and
\begin{equation}\label{coeffP}
c_{n-i}^P=(-1)^{n-i}\sum_{|\theta|=i}|\thetaA{P}|.
\end{equation}

Now, assume that $T$ is connected and loopless, expand these polynomials, compare $c_1^G$ with $c_1^P$ and apply
Kirchhoff's Matrix Tree Theorem and you will arrive that the following theorem, originally shown in \cite{RS74}. 
This is the version of the Matrix-Tree Theorem that will be used in this work. 

\begin{theorem}[Kirchhoff's Matrix-Tree Theorem for Probabilistic Graph Laplacians]\label{thm:matrixtree}  For any connected, loopless graph $T$ with $n$ labelled vertices, the number of spanning trees of $T$ is
$$\tau(T)=\left|\frac{\left(\prodd{j}{n}\right)}{\left(\sumd{j}{n}\right)}\left(\prodl{j}{n-1}{P}\right)\right|,$$
where $\{\lambda_j^P\}_{j=1}^{n-1}$ are the non-zero eigenvalues of $P$.
\end{theorem}

%\begin{proof}[Proof of Theorem~\ref{thm:matrixtree}]  From Kirchhoff's Matrix-Tree Theorem, we know that
%$$\tau(T)=\left|\frac{1}{n}c_1^G\right|.$$
%From Lemma~\ref{thm:congraph}, we have
%$$c_{1}^G=n\cdot (-1)^{1-n}\frac{\displaystyle \left(\prod_{j=1}^{n}\deg(v_j)\right)}{\left(\displaystyle \sum_{j=1}^{n}\deg(v_j)\right)}\cdot c_1^P.$$
%Also, we know $c_1^P=(-1)^{n-1}\prod_{j=1}^{n-1}\lambda_j^P$.
%So we have that
%$$\tau(T)=\left|\frac{1}{n}c_1^G\right|=\left|\frac{\left(\prodd{j}{n}\right)}{\left(\sumd{j}{n}\right)}\left(\prodl{j}{n-1}{P}\right)\right|.$$
%\end{proof}

\subsection{Fractal Graphs}
  Let $(X,d)$ be a complete metric space. If $f_i: X \rightarrow X$ is a contraction with respect to the metric $d$ for $i=1,2,...\ m,$ then there exist a unique non-empty compact subset $K$ of $X$ that satisfies 
$$K=f_1(K)\cup \cdot \cdot \cdot \cup f_m(K).$$
K is called the $self$-$similar\ set\ with\ respect\ to\ $ $\{f_1,f_2,...f_m\}$.

 If each $f_i$ is injective and for any n and for any two distinct words $\omega$, $\omega '$ $\in W_n$=$\{1,...m\}^n$ we have 
$$K_{\omega}\cap K_{\omega '}= F_{\omega}\cap F_{\omega '} $$
where $f_{\omega}$=$f_{\omega_1}\circ\cdot\cdot\cdot\circ f_{\omega_n}$,  
$K_{\omega}$=$f_{\omega}(K)$, $\ F_0$ is the set of fixed points of $\{f_1,f_2,...f_m\}$, and $F_{\omega}=f_{\omega}(F_0)$, is called a $finitely\ ramified$ $self$-$similar\ set\ with\ respect\ to\ \{f_1,f_2,...f_m\}$

For any self-similar set, K, with respect to $\{f_1,f_2,...f_m\}$. There is a natural  sequence of $approximating\ graphs\ V_n$ with vertex set $F_n$ defined as follows. For all $n\geq0$ and for all 
$\omega\in W_n$ define $V_{0}$ as the complete graph with vertices $F_{0}$, 
$$F_n:= \bigcup_{\omega\in W_n}F_{\omega},$$
$$F_{\omega}:= \bigcup_{x \in V_)}F_{\omega}(x),$$

where $F_{\omega}:=f_{a_n}\circ f_{a_{n-1}}\circ\cdots f_{a_1}\circ $ and $\omega=a_1a_2\cdots a_n$. Also, $x,y\in F_n$ are connected by an edge in $V_n$ if $f_i^{-1}(x)$ and $f_i^{-1}(y)$ are connected by an edge in $V_{n-1}$ for some $1\leq i\leq m$.

Let $K$ be a compact metrizable topological space and $S$ be a finite set. Also, let $F_i$ be a continuous injection from $K$ to itself $\forall i\in S$. Then, $(K,S,\{F_i\}_{i\in S})$ is called a $self$-$similar\ structure$ if there exists a continuous surjection $\pi:\Sigma\rightarrow K$ such that $F_i\circ\pi=\pi\circ\sigma_i$ $\forall i\in S$, where $\Sigma=S^{\NN}$ the one-sided infinite sequences of symbols in $S$ and $\sigma_i:\Sigma\rightarrow\Sigma$ is defined by $\sigma_i(\omega_1\omega_2\omega_3...)=i\omega_1\omega_2\omega_3...$ for each $\omega_1\omega_2\omega_3...\in\Sigma$

Clearly, if $K$ is the self-similar set with respect to injective contractions $\{f_1,f_2,...f_m\}$, then $(K,\{1,2,...m\},\{f_i\}_{i=1}^m)$ is a self-similar structure. 

Notice that two non-isomorphic self-similar structures can have the same finitely ramified self-similar set, however the structures will not have the same sequence of approximating graphs $V_n$. Also, any two isomorphic self-similar structures whose compact metrizable topological spaces are finitely ramified self-similar sets will have approximating graphs which are isomorphic $\forall n\geq0$.

A $fully\ symmetric$ finitely ramified self-similar structure with respect to $\{f_1,f_2,...f_m\}$ is a self-similar structure $(K,\{1,2,...m\},\{f_1,f_2,...f_m\})$ such that $K$ is a finitely ramified self-similar set, and, as in \cite{MR2450694}, for any permutation $\sigma:F_0\rightarrow F_0$ there is an isometry $g_{\sigma}:K\rightarrow K$ that maps any $x\in F_0$ into $\sigma(x)$ and preserves the self-similar structure of $K$. This means that there is a map $\tilde{g_{\sigma}}:W_1\rightarrow W_1$ such that $f_i\circ g_{\sigma}=g_{\sigma}\circ f_{\tilde{g_{\sigma}}(i)}$ $\forall i\in W_1$. The group of isometries $g_{\sigma}$ is denoted $\mathfrak{G}$.

As in \cite{HST11}, the definition of a fully symmetric finitely ramified self-similar structure may be combined into one compact definition. 

\begin{definition}
A fractal $K$ is a fully symmetric finitely ramified self-similar set if $K$ is a compact connected metric space with injective contraction maps on a complete metric space $\{f_i\}_{i=1}^m$ such that 
$$K=f_1(K)\cup \cdot \cdot \cdot \cup f_m(K).$$
and the following three conditions hold:
\begin{enumerate}
\item there exist a finite subset $F_0$ of $K$ such that
$$f_j(K)\cap f_k(K)=f_j(F_0)\cap f_k(F_0)  $$
for $j\neq k$ (this intersection may be empty);
\item if $v_0\in F_0\cap f_j(K)$ then $v_0$ is the fixed point of $f_j$;
\item there is a group $\mathcal{G}$ of isometries of $K$ that has a doubly transitive action on $F_0$ and is compatible with the self-similar structure $\{f_i\}_{i=1}^m$, which means that for any $j$ and any $g\in \mathcal{G}$ there exist a $k$ such that
$$g^{-1}\circ f_j\circ g=f_k. $$
\end{enumerate}

\end{definition}

\section{Counting Spanning Trees on Fractal Graphs}~\label{ch:mainresult}

Let $K$ be a fully symmetric finitely ramified self-similar structure, $V_n$ be its sequence of approximating graphs, and $P_n$ denote the probabilistic graph Laplacian of $V_n$. 

The next two Propositions describe the spectral decimation process, which inductively gives the spectrum of $P_n$. 

The $V_0$ network is the complete graph on the boundary set and we set $m=|V_0|$. Write $P_1$ in block form
$$P_1=\begin{pmatrix}A&B\\
C&D\\
\end{pmatrix}$$ 
where A is a square block matrix associated to the boundary points. Since the $V_1$ network never has an edge joining two boundary points A is the $mxm$ identity matrix. 
The Schur Complement of $P_1$ is 
\begin{equation*}
 S(z)=(A-zI)-B(D-z)^{-1}C
\end{equation*}
 
\begin{proposition}\label{prop:existR}(Bajorin, et al.,\cite{MR2450694}) For a given fully symmetric finitely ramified self-similar structure K there are unique scalar valued rational functions $\phi(z)$ and $R(z)$ such that for $z\notin\sigma(D)$
\begin{equation*}
 S(z)=\phi(z)(P_0-R(z))
\end{equation*}
Now $P_0$ has entries $a_{ii}=1$ and $a_{ij}=\frac{-1}{m-1}$ for $i\neq j$. Looking at specific entries of this matrix valued equation we get two scalar valued equations
\begin{equation*}
 \phi(z)=-(m-1)S_{1,2}(z)
\end{equation*}
and
\begin{equation*}
 R(z)= 1-\frac{S_{1,1}}{\phi(z)}.
\end{equation*} 
Where $S_{i,j}$ is the $i,j$ entry of the matrix $S(z)$. 
\end{proposition}

Now, we let 
\begin{equation*}
 E(P_0,P_1):=\sigma(D)\bigcup\{z:\phi(z)=0\}
\end{equation*} 
and call $E(P_0,P_1)$ the exceptional set. 
\\
Let $mult_D(z)$ be the multiplicity of $z$ as an eigenvalue of $D$, $mult_n(z)$ be the multiplicity of $z$ as an eigenvalue of $P_n$, $mult_n(z)=0$ if and only if $z$ is not an eigenvalue of $P_n$, and similarly $mult_D(z)=0$ if and only if $z$ is not and eigenvalue of $D$. Then we may inductively find the spectrum of $P_n$ with the following Proposition. 

\begin{proposition}\label{prop:inductionsteps}(Bajorin, et al.,\cite{MR2450694}) For a given fully symmetric finitely ramified self-similar structure K, and $R(z),\ \phi(z),\ E(P_0,P_1)$ as above, the spectrum of $P_n$ may be calculate inductively using the following criteria:
\begin{enumerate}
 \item if $z\notin E(P_0,P_1)$, then
	\begin{equation*}
 	mult_n(z)=mult_{n-1}(R(z))
	\end{equation*}
\item if $z\notin \sigma(D)$, $\phi(z)=0$ and $R(z)$ has a removable singularity at z then,
	\begin{equation*}
	 mult_n(z)=|V_{n-1}|
	\end{equation*}
\item if $z\in \sigma(D)$, both $\phi(z)$ and $\phi(z)R(z)$ have poles at z, $R(z)$ has a removable singularity at z, and $\frac{\partial}{\partial z}R(z)\neq 0$, then
	\begin{equation*}
	 mult_n(z)=m^{n-1}mult_D(z)-|V_{n-1}|+mult_{n-1}(R(z))
	\end{equation*}
 \item if $z\in \sigma(D)$, but $\phi(z)$ and $\phi(z)R(z)$ do not have poles at z, and $\phi(z)\neq 0$,then 
	\begin{equation*}
	 mult_n(z)=m^{n-1}mult_D(z)+mult_{n-1}(R(z))
	\end{equation*}
\item if $z\in \sigma(D)$, but $\phi(z)$ and $\phi(z)R(z)$ do not have poles at z, and $\phi(z)=0$,then 
	\begin{equation*}
	 mult_n(z)=m^{n-1}mult_D(z)+|V_{n-1}|+mult_{n-1}(R(z))
	\end{equation*}
\item if $z\in \sigma(D)$, both $\phi(z)$ and $\phi(z)R(z)$ have poles at z, $R(z)$ has a removable singularity at z, and $\frac{\partial}{\partial z}R(z)=0$, then
	\begin{equation*}
	  mult_n(z)=m^{n-1}mult_D(z)-|V_{n-1}|+2mult_{n-1}(R(z))
	\end{equation*}
\item if $z\notin \sigma(D)$, $\phi(z)=0$ and $R(z)$ has a pole at z, then $mult_n(z)=0$.
\item if $z\in\sigma(D)$, but $\phi(z)$ and $\phi(z)R(z)$ do not have poles at z, $\phi(z)=0$ and $R(z)$ has a pole at z, then 
	\begin{equation*}
	 mult_n(z)=m^{n-1}mult_D(z).
	\end{equation*}
\end{enumerate}
\end{proposition} 

After carrying out the inductive calculations using items (1)-(8), define 
\begin{equation*}
 A:=\{\alpha : \alpha \textrm{ satisfies item (2) or (8)}\}
\end{equation*}
for $\alpha\in A$, $\alpha_n:=mult_n(\alpha)$
\begin{equation*}
 B:=\{\beta : \textrm{ for some }n\geq1,\  mult_n(\beta) \neq0 \textrm{ and } mult_{n-1}(R(\beta))\neq0\}
\end{equation*}
and for $\beta \in B$, $\beta_n^k:=mult_n(R_{(-k)}(\beta))$.

Since $V_n$ is connected $mult_n(0)=1$ for all $n\geq0$. Again from \cite{MR2450694}, we get that 

\begin{equation*}
\sigma(P_n)\setminus\{0\}=\bigcup_{\alpha\in A} \left\{\alpha\right\} \bigcup_{\beta\in B}\left[\bigcup_{k=0}^{n}\bigl\{R_{-k}(\beta):\beta_n^k\neq0\bigr\}\right].
\end{equation*}
\\

Hence the non-zero eigenvalues of $P_n$ are the zeros of polynomials or preiterates of rational functions. 
To be able to use Theorem~\ref{thm:matrixtree}, we need to know how to take the product of preiterates of 
rational functions of a particular form. The proof of Theorem~\ref{thm:maintheoremfull} will show that 
$R(z)$ satisfies the assumptions of the next Lemma, and use this information to be able to calculate the
number of spanning trees on the fractal graphs under consideration.

\begin{lemma}\label{prop:rationals}  Let $R(z)$ be a rational function such that $R(0)=0$, $\deg(R(z))=d$, $R(z)=\frac{P(z)}{Q(z)}$, with $\deg(P(z))>\deg(Q(z))$.  Let $P_d$ be the leading coefficient of $P(z)$.  Fix $\alpha\in\CC$.  Let $\{R_{(-n)}(\alpha)\}$ be the set of $n^{\text{th}}$ preiterates of $\alpha$ under $R(z)$.  By convention, $R_{(0)}(\alpha):=\{\alpha\}$.  Then for $n\geq 0$,
$$\prod_{z\in\{R_{(-n)}(\alpha)\}}z=\alpha\left(\frac{-Q(0)}{P_d}\right)^{\left(\frac{d^n-1}{d-1}\right)}.$$
\end{lemma}
\begin{proof}[Proof of Lemma~\ref{prop:rationals}] 
For $n=0$, the result is clear.  For $n=1$, we note
\begin{align*}\{R_{(-1)}(\alpha)\}&=\{z:R(z)=\alpha\}\\
&=\{z:P(z)-\alpha Q(z)=0\}\\
&=\{z:P_dz^d+\cdots-Q(0)\alpha=0\},
\end{align*}
where $Q(0)$ is the constant term of $Q(z)$.  As the product of the roots of a polynomial is equal to the constant term over the coefficient of the highest degree term, we have that
$$\prod_{z\in\{R_{(-1)}(\alpha)\}}z=\frac{-\alpha Q(0)}{P_d}.$$
Assume our equation holds for $n$.  Then for $n+1$ we have
$$\bigl\{w:w\in R_{(-(n+1))}(\alpha)\bigr\}=\bigl\{R_{(-1)}(w):w\in R_{(-n)}(\alpha)\bigr\}.$$
Hence,
\begin{align*}
\prod_{w\in\left\{R_{(-(n+1))}(\alpha)\right\}}w&=\prod_{w\in\left\{R_{(-n)}(\alpha)\right\}}\left(\prod_{z\in\left\{R_{(-1)}(w)\right\}} z\right)
%&
=\prod_{w\in\left\{R_{(-n)}(\alpha)\right\}}\left(\frac{-wQ(0)}{P_d}\right),
\end{align*}
with the second equality following from the $n=1$ case.\\
\\
Since $\left|R_{(-n)}(\alpha)\right|=d^n$ (not necessarily distinct) this equality becomes
\begin{align*}
\prod_{w\in\left\{R_{(-(n+1))}(\alpha)\right\}}w&=\left(\frac{-Q(0)}{P_d}\right)^{d^n}\prod_{w\in\left\{R_{(-n)}(\alpha)\right\}}w\\
&=\left(\frac{-Q(0)}{P_d}\right)^{d^n}\cdot (\alpha)\left(\frac{-Q(0)}{P_d}\right)^{\left(\frac{d^n-1}{d-1}\right)}\\
&=\alpha\left(\frac{-Q(0)}{P_d}\right)^{\left(\frac{d^{n+1}-1}{d-1}\right)},
\end{align*}
as desired.
\end{proof}

The following Theorem is the main result of this paper.

%new proof of main theorem with old latex code commented out
\begin{theorem}\label{thm:maintheoremfull}  For a given fully symmetric self-similar structure on a finitely ramified fractal K, let $V_n$ denote its sequence of approximating graphs and let $P_n$ denote the probabilistic graph Laplacian of $V_n$. Arising naturally from the spectral decimation process, there is a rational function $R(z)$, which satisfies the conditions of Lemma~\ref{prop:rationals}, finite sets $A,B\subset\RR$ such that for all $\alpha\in A$, $\beta\in B$, and integers $n,k\geq 0$, there exist functions $\alpha_n$ and $\beta_n^k$ such that the number of spanning trees of $V_n$ is given by\\

\begin{equation}\label{eqn:maintheoremfull}
\begin{split}
\tau(V_n)
% 
%The following lines of code were deleted to make the theorem easier to read. Can readd them in
%  by deleting the next few ``%''s
%
%&=\left|\frac{\left(\prodd{j}{|V_n|}\right)}{\left(\sumd{j}{|V_n|}\right)}\left(\prod_{\alpha\in A}\alpha^{\alpha_n}\right)\left[\prod_{\beta\in B}\left(\prod_{k=0}^{n}\left(\beta\left(\frac{-Q(0)}{P_d}\right)^{\frac{d^k-1}{d-1}}\right)^{\beta_n^k}\right)\right]\right|\\
%&
=\left|\frac{\left(\prodd{j}{|V_n|}\right)}{\left(\sumd{j}{|V_n|}\right)}\left(\prod_{\alpha\in A}\alpha^{\alpha_n}\right)\left[\prod_{\beta\in B}\left(\beta^{\sumw{k}{n}{\beta}}\left(\frac{-Q(0)}{P_d}\right)^{\sum_{k=0}^n\beta_n^k\left(\frac{d^k-1}{d-1}\right)}\right)\right]\right|\\
\end{split}
\end{equation}
where $d$ is the degree of $R(z)$, $P_d$ is the leading coefficient of the numerator of $R(z)$, $|V_n|$ is the number of vertices of $V_n$ and $d_j$ is the degree of vertex $j$ in $V_n$.
\end{theorem}
\begin{proof}[Proof of Theorem~\ref{thm:maintheoremfull}]  
From Kirchhoff's matrix-tree theorem for probabilistic graph Laplacians (Theorem ~\ref{thm:matrixtree}), we know that
\begin{equation*}
\tau(V_n)=\frac{\prodd{j}{|V_n|}}{\sumd{j}{|V_n|}} \prod_{j=1}^{|V_n|-1}\lambda_j
\end{equation*}
where $\lambda_j$ are the non-zero eigenvalues of $P_n$.

Existence and uniqueness of the rational function $R(z)$ is given Proposition (\ref{prop:existR}). After carrying out the inductive calculations using Proposition (\ref{prop:inductionsteps}) items (1)-(8), we get the sets $A$ and $B$, and the functions $\alpha_n$ and $\beta_n^k$. 
\\
\\
To see that the sets $A$ and $B$ are finite. Recall that the functions $R(z)$ and $\phi(z)$ from Proposition (\ref{prop:inductionsteps}) are rational, thus $R(z)$, $\phi(z)$, and $R(z)\phi(z)$ have finitely many zeroes, poles, and removable singularities. Also, since the matrix $D$, from writing $P_1$ in block form to define the Schur Complement, is finite, $\sigma(D)$ is finite. Following items (1)-(8) of Proposition (\ref{prop:inductionsteps}) these observations imply that $A$ and $B$ are finite sets.
\\
\\
From Proposition (\ref{prop:inductionsteps}) we know that 
\begin{equation*}
\bigl\{\lambda_j\bigr\}_{j=1}^{|V_n|-1}=\bigcup_{\alpha\in A} \left\{\alpha\right\} \bigcup_{\beta\in B}\left[\bigcup_{k=0}^{n}\bigl\{R_{-k}(\beta):\beta_n^k\neq0\bigr\}\right]
\end{equation*}
where the multiplicities of $\alpha\in A$ are given by $\alpha_n$ and the multiplicities of $\{R_{-k}(\beta)\}$ are given by $\beta_n^k$. Letting $\lambda_{|V_n|}=0$. 
\\
\\
From items (1)-(8) of Proposition (\ref{prop:inductionsteps}) it follows that $\forall z\in \{R_{-k}(\beta)\}$ the multiplicity of z depends only on $n$ and $k$, thus
\\
\begin{equation*}
\prod_{j=1}^{|V_n|-1}\lambda_j=\left(\prod_{\alpha\in A}\alpha^{\alpha_n}\right)\left[\prod_{\beta\in B}\left(\prod_{k=0}^{n}\left(
\prod_{z\in\{R_{-k}(\beta)\}}
z^{\beta_n^k}\right)\right)\right]
\end{equation*}

From Lemma 4.9 in \cite{MR1997913}, $R(0)=0$. From Corollary 1 in \cite{HST11}, it follows that, if we write $R(z)=\frac{P(z)}{Q(z)}$ where $P(z)$ and $Q(z)$ are relatively prime polynomials, then $deg(P(z))>deg(Q(z))$. Thus, the conditions of Lemma~\ref{prop:rationals} are satisfied, and applying this theorem gives
\\
\begin{equation*}
=\left(\prod_{\alpha\in A}\alpha^{\alpha_n}\right)\left[\prod_{\beta\in B}\left(\prod_{k=0}^{n}\left(\beta\left(\frac{-Q(0)}{P_d}\right)^{\frac{d^k-1}{d-1}}\right)^{\beta_n^k}\right)\right]
\end{equation*}
\\
\begin{equation*}
=\left(\prod_{\alpha\in A}\alpha^{\alpha_n}\right)\left[\prod_{\beta\in B}\left(\beta^{\sumw{k}{n}{\beta}}\left(\frac{-Q(0)}{P_d}\right)^{\sum_{k=0}^n\beta_n^k\left(\frac{d^k-1}{d-1}\right)}\right)\right]
\end{equation*}
Applying Kirchhoff's matrix-tree theorem for probabilistic graph Laplacians (Theorem~\ref{thm:matrixtree}), we verify the result. 

\end{proof}

Section \ref{section:ex} of this work will begin with a well known example, the $\sierp$ Gasket, and show how to use this theorem
to calculate the number of spanning trees on the fractal graphs under consideration.  This theorem will then
be used to compute the number of spanning trees for three previously unknown examples. 

In \cite{CCY07}, the authors derived multidimensional polynomial recursion equations to solve explicity for
the number of spanning trees on $SG_d(n)$ with $d$ equal to two, three and four, and on $SG_{d,b}(n)$ with
$d$ equal to two and $b$ equal to two and three. They note in that work that it is intriguing that their 
recursion relations become more and more complicated as $b$ and $d$ increase, but the solutions remain simple, 
and comment that with their methods, they do not have a good explanation for this. The following
corollary explains why the solutions remain simple. 

\begin{corollary}\label{cor:simple}
For a given fully symmetric self-similar structure on a finitely ramified fractal K, with approximating graphs
$V_n$, there exist a finite set of primes $\{p_k\}_{k=1}^r$ and functions $\{f_k:\mathbb{N}_0\rightarrow \mathbb{N}_0\}_{k=1}^r$
such that
$$\tau(V_n)=\prod_{k=1}^{r}p_k^{f_k(n)}.$$
\end{corollary}

\begin{proof}[Proof of Corollary \ref{cor:simple}] 
Since $\tau(V_n)$ is a nonnegative integer, and is given by equation (\ref{eqn:maintheoremfull}), the sets
$A$ and $B$ are fixed,  and self-similarity gives that for any $n\geq2$ the only prime factors of 
$\left(\prod_{i=1}^{|V_n|}d_i\right)$ are the prime factors of $\left(\prod_{i=1}^{|V_1|}d_i\right)$.
\end{proof}

%\idiot{consider the placement of this proposition}

\section{Asymptotic Complexity}

Let $T_n$ for $n\geq0$ be a sequence of finite graphs, $|T_n|$ the number of vertices in $T_n$, and $\tau(T_n)$ denote the number of spanning trees of $T_n$. $\tau(T_n)$ is called the $complexity$ of  $T_n$. The $asymptotic\ complexity$ of the sequence $T_n$ is defined as
$$\lim_{n\rightarrow\infty} \frac{log(\tau(T_n))}{|T_n|}.$$
When this limit exist, it is called the $asymptotic\ complexity\ constant$, or the $tree\ entropy$ of $T_n$, or the $thermodynamic\ limit$ of $T_n$.

%\begin{proposition}{\label{lemma:wedge}}
For any two, finite, connected graphs $G_1$, $G_2$, let $G_1\vee_{x_1,x_2} G_2$ denote the graph formed by identifying the vertex $x_1\in G_1$ with vertex $x_2\in G_2$. Then $\forall x_1\in G_1, x_2\in G_2$, it is clear that
\begin{equation}{\label{eqn:wedge}}
 \tau(G_1\vee_{x_1,x_2} G_2)=\tau(G_1)\cdot \tau(G_2).
\end{equation}
%\end{proposition}

%\begin{proof}[Proof of Proposition ~\ref{lemma:wedge}]
%Any spanning tree of $G_1\vee_{x_1,x_2} G_2$ when restricted to $G_1$ is a spanning tree of $G_1$, and similarly for $G_2$, so $$\tau(G_1\vee_{x_1,x_2} G_2)\leq \tau(G_1)\cdot \tau(G_2).$$ For any spanning trees $T_1,T_2$ of $G_1$ and $G_2$ respectively, $T_1\vee_{x_1,x_2} T_2$ is a spanning tree of $G_1\vee_{x_1,x_2} G_2$. This gives
%$$\tau(G_1\vee_{x_1,x_2} G_2)\geq \tau(G_1)\cdot \tau(G_2),$$ as desired. 
%\end{proof}

Dropping  the assumption of full symmetry, we lose the spectral decimation process, but still have the following. 

\begin{theorem}\label{thm:asycomplexity} For a given self-similar structure on a finitely ramified fractal K, let $V_n$ denote its sequence of approximating graphs. Let $m$ denote the number of 0-cells of the $V_1$ graph.
\begin{enumerate}
 \item If $V_1$ is a tree, then $\tau(V_n)=1$ $\forall n\geq0$
 \item If $V_1$ is not a tree, then $log(\tau(V_n))\in \theta(|V_n|)=\theta(m^n)$
\end{enumerate}
\end{theorem}

\begin{proof}[Proof of Theorem~\ref{thm:asycomplexity}]
If $V_1$ is a tree, then K is a fractal string. Hence $\forall n \geq0$ $V_n$ is a tree.
If $V_1$ is not a tree, it is $m$ copies of the $V_0$ graph with vertices identified appropriately. Similarly the $V_n$ graph is $m^n$ copies of the $V_0$ graph with vertices identified appropriately. Let $V_0\vee_{x,x}^{m^n}V_0$ denote $m^n$ copies of $V_0$ each identified to each other at some vertex $x\in V_0$, then clearly for $n\geq0$
\begin{equation}
 \tau(V_n)\geq \tau(V_0\vee_{x,x}^{m^n}V_0).
\end{equation}

Since $V_1$ is not a tree, $|V_0|>2$, also the $V_0$ graph is the complete graph on $|V_0|$ vertices, so by Cayley's formula, \cite{MR1813436}, $\tau(V_0)=|V_0|^{(|V_0|-2)}$.

Combining this with equation (\ref{eqn:wedge}) we get that $$\tau(V_0\vee_{x,x}^{m^n}V_0)=|V_0|^{(|V_0|-2)\cdot m^n}$$ and 
$$ \tau(V_n)\geq |V_0|^{(|V_0|-2)\cdot m^n}.$$
So for $n\geq0$, 
\begin{equation}\label{eqn:lowerasy}
 log(\tau(V_n))\geq m^n\cdot (|V_0|-2)log(|V_0|)\sim |V_n|
\end{equation}
Since $m^n\sim|V_n|$

Now, $V_n$ can also be constructed by deletion of edges from the graph $K_{|V_n|}$. The deletion-contraction principle, \cite{MR1813436}, says that for any connected graph $G$ and any edge $e$ in that graph 
$$\tau(G)=\tau(G\backslash e) + \tau(G-e),$$ 
where $G\backslash e$ is the graph formed by contracting $e$ in $G$ and $G-e$ is the graph formed by deleting $e$ from $G$. 

This tells us that deleting edges from graphs decreases the number of spanning trees, thus 
$$\tau(V_n)\leq \tau(K_{|V_n|})=|V_n|^{(|V_n|-2)}.$$
Since $|V_n|\sim m^n$, $$\tau(V_n)\lesssim m^{n(m^n-2)},$$ which implies $\forall \epsilon>0$ 
\begin{equation}\label{eqn:upperepsasy}
 \lim_{x \to \infty}\frac{log(\tau(V_n))}{m^{n(1+\epsilon)}}=0.
\end{equation}
Now, suppose that the sequence $\frac{log(\tau(V_n))}{m^n}$ is unbounded then $\forall M>0$ $\exists n_0$ s.t. $\forall n\geq n_0$ $\frac{log(\tau(V_n))}{m^n} >M$, but then $\forall \epsilon >0$ and $\forall n> \frac{n_0}{(1+\epsilon)}$, $\frac{log(\tau(V_n))}{m^{n(1+\epsilon)}}>M$ which contradicts equation (\ref{eqn:upperepsasy}). 
Thus, $\frac{log(\tau(V_n))}{m^n}$ is bounded and combining this with equation (\ref{eqn:lowerasy}) implies $\log(\tau(V_n))\in \theta(|V_n|)$, as desired. 

\end{proof}

\begin{corollary}\label{cor:asycomplexity} For a given self-similar structure on a finitely ramified fractal K, let $V_n$ denote its sequence of approximating graphs. The following limits exist.

\begin{equation}
 \limsup_{n\rightarrow\infty} \frac{log(\tau(V_n))}{|V_n|},
\end{equation}
\begin{equation}
 \liminf_{n\rightarrow\infty} \frac{log(\tau(V_n))}{|V_n|}.
\end{equation}

\end{corollary}

We conclude this section with a few conjectures. The first is the natural question to follow from
Corollary~\ref{cor:asycomplexity}. Is additionally requiring full symmetry enough to get convergence?

\begin{conjecture}\label{conj:asycomplexity}For a given fully symmetric self-similar structure on a finitely ramified fractal K, let $V_n$ denote its sequence of approximating graphs. The following limits exist.
\begin{equation}
 \lim_{n\rightarrow\infty} \frac{log(\tau(V_n))}{|V_n|}. 
\end{equation}
 
\end{conjecture}

The family of fractal trees indexed by the number of branches they possess provide a nice class of examples, and is studied in \cite{FS10}.
These examples show that even though each $m$-Tree Fractal in the limit is topologically a tree, 
the number of spanning trees on the approximating graphs grows arbitrarily large. Now, considering 
the $3$-Tree Fractal with graph approximations $V_{3,n}$, using equation (\ref{eqn:wedge}) it is 
easy to verify that $\tau(V_{3,n})=3^{3^n}$ for $n\geq0,$ and the the asymptotic complexity constant is 
$\frac{log(3)}{2}.$ 

\begin{conjecture} For a given fully symmetric self-similar structure on a finitely ramified fractal K, 
let $V_n$ denote its sequence of approximating graphs, and $c_K$ denote the asymptotic complexity constant.
If $V_1$ is not a tree, then $$c_K\geq\frac{log(3)}{2}.$$
 \end{conjecture}

%\idiot{Remember to include the unboundedness of the asymptotic complexity constant, if desired}

\section{Examples}\label{section:ex}

\subsection{$\sierp$ Gasket} \label{Section: sierp}

The $\sierp$ gasket has been extensively studied (in \cite{St06, MR2451619, Ki01, Ra84, Be92, DSV99, FS92, Sh96, Te98}, among others.)  It can be constructed as a p.c.f. fractal, in the sense of Kigami \cite{Ki01}, in $\RR^2$ using the contractions
\begin{align*}
f_i(x)&=\frac{1}{2}(x-q_i)+q_i,
\end{align*}
for $i=1,2,3$, where the points $q_i$ are the vertices of an equilateral triangle.

%%%%
% V_1 of the sierp gasket
%%%%

\begin{figure}[h!]
\begin{center}
\epsfig{file=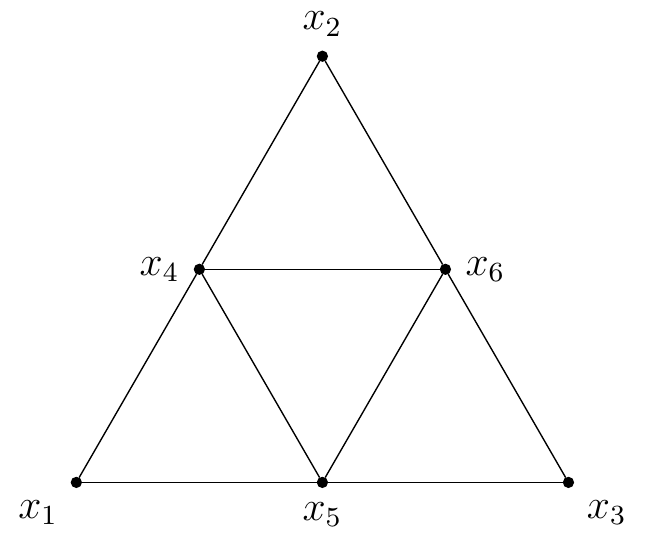, scale=.6}

\label{fig:pcfsierp1}
\caption{The $V_1$ network of $\sierp$ gasket.}
\end{center}

\end{figure}

In \cite{CCY07}, the following theorem was proven. Here we give a new proof using the 
method described in Section~\ref{ch:mainresult} to show how to use Theorem ~\ref{thm:maintheoremfull}.

\begin{theorem}\label{thm:spansierp}  The number of spanning trees on the $\sierp$ gasket at level $n$ is given by
$$\tau(V_n)=2^{f_n}\cdot 3^{g_n}\cdot 5^{h_n},\;\;\;\;\;\;n\geq 0$$
where
\begin{align*}
f_n&=\frac{1}{2}\left(3^n-1\right),
g_n=\frac{1}{4}\left(3^{n+1}+2n+1\right), \text{ and }
h_n=\frac{1}{4}\left(3^n-2n-1\right).\\
\end{align*}
\end{theorem}
\begin{proof}[Proof of Theorem~\ref{thm:spansierp}]  Before applying Theorem~\ref{thm:maintheoremfull}, we make the
following observations.
It is well known that the $V_n$ network of the $\sierp$ gasket has
$$|V_n|=\frac{3^{n+1}+3}{2}\;\;\;\;\;n\geq 0$$
vertices, three of which have degree 2 and the remaining vertices have degree 4.  
Hence,
\begin{equation}\label{eqn:sierp3}
\frac{\displaystyle \prodd{i}{|V_n|}}{\displaystyle \sumd{i}{|V_n|}}=2^{3^{n+1}-1}\cdot 3^{-(n+1)}.
\end{equation}
In \cite{MR2450694}, they use a result from \cite{MR2451619} to carry out spectral decimation for the $\sierp$ gasket.  In our language, they showed that
\begin{align*}
A&=\left\{\frac{3}{2}\right\},
B=\left\{\frac{3}{4},\frac{5}{4}\right\},
\end{align*}

\begin{enumerate}[(I)]
\item\label{sfirst} $\alpha=\frac{3}{2}$, $\alpha_n=\frac{3^n+3}{2},\;\;\;\;\;n\geq 0$,
\item\label{ssecond} $\beta=\frac{3}{4}$,\;\;\;\;\; $n\geq 1$
\begin{equation*}\beta_n^k=\begin{cases} \frac{3^{n-k-1}+3}{2}&\;\;\;\;\;k=0,\ldots,n-1\\
0&\;\;\;\;\;k=n,\\
\end{cases}
\end{equation*}
\item\label{sthird} $\beta=\frac{5}{4}$,\;\;\;\;\; $n\geq 2$
\begin{equation*}\beta_n^k=\begin{cases} \frac{3^{n-k-1}-1}{2}&\;\;\;\;\;k=0,\ldots,n-2\\
0&\;\;\;\;\;k=n-1,n\\
\end{cases}
\end{equation*}
\end{enumerate}
and $R(z)=z(5-4z)$.  So $d=2$, $Q(0)=1$ and $P_d=-4$.\\
\\
We now use Equation~\ref{eqn:maintheoremfull} in Theorem~\ref{thm:maintheoremfull} to calculate $\tau(V_n)$.  We have
\begin{equation}\label{eqn:sierpalpha}
\prod_{\alpha\in A} \alpha^{\alpha_n}=\left(\frac{3}{2}\right)^{\displaystyle \frac{3^n+3}{2}}
\end{equation}
\begin{equation}\label{eqn:sierpbeta}
\begin{split}
\prod_{\beta\in B}&\left(\beta^{\sum_{k=0}^n\beta_n^k}\cdot\left(\frac{1}{4}\right)^{\sum_{k=0}^n\beta_n^k\left(2^k-1\right)}\right)=\\
&=\left(\frac{3}{4}\right)^{\displaystyle \sum_{k=0}^{n-1}\left(\frac{3^{n-k-1}+3}{2}\right)}\times\left(\frac{1}{4}\right)^{\displaystyle \sum_{k=0}^{n-1}\left(\frac{3^{n-k-1}+3}{2}\right)\left(2^k-1\right)}\\
&\times\left(\frac{5}{4}\right)^{\displaystyle \sum_{k=0}^{n-2}\left(\frac{3^{n-k-1}-1}{2}\right)}\times\left(\frac{1}{4}\right)^{\displaystyle \sum_{k=0}^{n-2}\left(\frac{3^{n-k-1}-1}{2}\right)\left(2^k-1\right)}\\
\end{split}
\end{equation}

We sum the expressions in the exponents above.
\begin{align*}
\sum_{k=0}^{n-1}\left(\frac{3^{n-k-1}+3}{2}\right)&=\frac{1}{4}\left(3^n+6n-1\right)\\
\sum_{k=0}^{n-1}\left(\frac{3^{n-k-1}+3}{2}\right)\left(2^k-1\right)&=\frac{1}{4}\left(3^n+2^{n+2}-6n-5\right)\\
\sum_{k=0}^{n-2}\left(\frac{3^{n-k-1}-1}{2}\right)&=\frac{1}{4}\left(3^n-2n-1\right)\\
\sum_{k=0}^{n-2}\left(\frac{3^{n-k-1}-1}{2}\right)\left(2^k-1\right)&=\frac{1}{4}\left(3^n-2^{n+2}+2n+3\right).
\end{align*}
All of these equations are valid for $n\geq 2$.  Using equations~\ref{eqn:maintheoremfull},~\ref{eqn:sierp3},~\ref{eqn:sierpalpha},and~\ref{eqn:sierpbeta}, and simplifying we get:
$$\tau(V_n)=2^{f_n}\cdot 3^{g_n}\cdot 5^{h_n}\;\;\;\;\;n\geq 2,$$
as desired.For $n=1$, equation~\ref{eqn:sierp3} still holds and the eigenvalues of the probabilistic graph Laplacian are $\{\frac{3}{2},\frac{3}{2},\frac{3}{2},\frac{3}{4},\frac{3}{4},0\}.$  So by Theorem~\ref{thm:matrixtree}, we get that $\tau(V_1)=2\cdot 3^3$.  The $V_0$ network is the complete graph on 3 vertices, thus $\tau(V_0)=3$. Hence the theorem holds for all $n\geq 0$.
\end{proof}

As in \cite{CCY07}, we immediately have the following Corollary.

\begin{corollary} The asymptotic growth constant for the $\sierp$ Gasket is 
\begin{equation}
c=\frac{log(2)}{3}+\frac{log(3)}{2}+\frac{log(5)}{6}
 \end{equation}
\end{corollary}

%\begin{proof}
%Use Theorem \ref{thm:spansierp} and recall that 
%$$|V_n|=\frac{3^{n+1}+3}{2}\;\;\;\;\;n\geq 0$$

%\end{proof} 

\subsection{A Non-p.c.f. Analog of the Sierpi\'{n}ski Gasket}\label{section:nonpcf}
As described in \cite{MR2451619, BST99,Te08}, this fractal is finitely ramified by not p.c.f. in the sense of Kigami.  It can be constructed as a self-affine fractal in $\RR^2$ using 6 affine contractions.  One affine contraction has the fixed point $(0,0)$ and the matrix
\begin{equation*}
\begin{pmatrix}
\frac{1}{2}&\frac{1}{6}\\
\frac{1}{4}&\frac{1}{4}
\end{pmatrix},
\end{equation*}
and the other five affine contractions can be obtained though combining this one with the symmetries of the equilateral triangle on vertices $(0,0)$, $(1,0)$ and $\left(\frac{1}{2},\frac{\sqrt{3}}{2}\right)$.  Figure~\ref{fig:sierp1} shows the $V_1$ network for this fractal.

%%%% Code for the figure

\begin{figure}[h!]
%\begin{center}
\centering
\epsfig{file=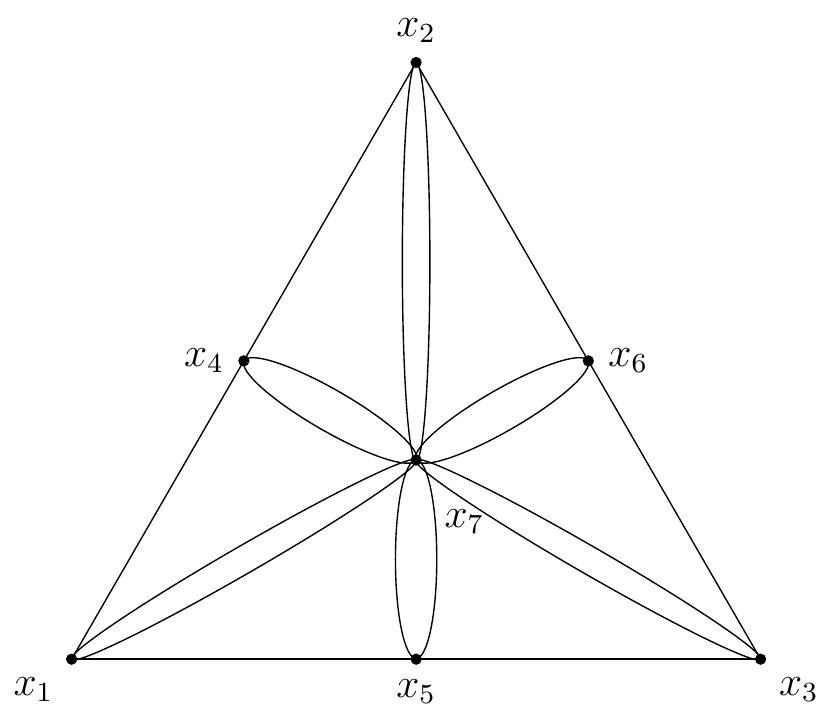, scale=.5}
\label{fig:sierp1}
\\
\caption{The $V_1$ network of the non-p.c.f. analog of the Sierpi\'{n}ski gasket.}

%\end{center}
\end{figure}

\begin{theorem}\label{thm:nonpcf}  The number of spanning trees on the non-p.c.f. analog of the Sierpi\'{n}ski gasket at level n is given by

$$\tau(V_n)=2^{f_n}\cdot 3^{g_n}\cdot 5^{h_n},\;\;\; n\geq 0$$
where
%This is the code to put the equations as they are in my thesis.
%$$f_n&=\frac{2}{25}\left(11\cdot 6^n-30n-11\right),\\
%g_n&=\frac{1}{5}\left(2\cdot 6^n+3\right),\text{ and }\\
%h_n&=\frac{1}{25}\left(4\cdot 6^n+30n-4\right).$$

\begin{align*}
f_n&=\frac{2}{25}\left(11\cdot 6^n-30n-11\right),\text{ }
g_n=\frac{1}{5}\left(2\cdot 6^n+3\right),\text{ and}\\
h_n&=\frac{1}{25}\left(4\cdot 6^n+30n-4\right).
\end{align*}
\end{theorem}
Before the proof, we need a few results.
\begin{lemma}\label{thm:vertexdegrees1}  The $V_n$ network of the non-p.c.f. analog of the Sierpi\'{n}ski gasket, for $n\geq 0$, has
$$\frac{4\cdot 6^n+11}{5}$$
vertices.  Among these vertices,
\begin{enumerate}[(i)]
\item $3$ have degree $2^{n+1}$,
\item $6^{k-1}$ have degree $3\cdot 2^{n-k+2}$ for $1\leq k\leq n$, and
\item $3\cdot 6^{k-1}$ have degree $2^{n-k+2}$ for $1\leq k\leq n$.
\end{enumerate}
\end{lemma}
\begin{proof}[Proof of Lemma~\ref{thm:vertexdegrees1}] 
We first describe how the $V_n$ network is constructed, then prove the Lemma. 

For $n=0$, $V_0$ is the complete graph on vertices \{$x_1,x_2,x_3$\}, one triangle (the $V_0$ network) and 3 corners of degree 2 \{$x_1,x_2,x_3$\} are born at level 0. 

For $n=1$, from the triangle born on level 0, 6 triangles are born. For example one of these triangles is the complete graph on \{$x_2,x_4,x_7$\}. 3 corners of degree 4 are born, they are \{$x_4,x_5,x_6$\} and one center is born \{$x_7$\} of degree 12. 

For $n\geq2$, from each triangle born at level $n-1$, 6 triangles are born, 3 corners of degree 4 are born and 1 center of degree 12 is born. Each corner born at level $n-1$ gains 4 edges. Each center born at level $n-1$ gains 12 edges. Each corner born at level $n-2$ gains $2\cdot4$ edges. Each center born at level $n-2$ gains $2\cdot12$ edges. In general, for $1\leq k\leq n-1$, each corner born at level $n-k$ gains $2^{k-1}\cdot4$ edges, and each center born at level $n-k$ gains $2^{k-1}\cdot12$ edges. The corners born at level $0$ gain $2^n$ edges. 

From this construction we see that, for $n\geq0$ the $V_n$ network has 
\begin{equation*}
 3+4\cdot\sum_{j=0}^{n-1}6^j= \frac{4\cdot6^n+11}{5}
\end{equation*}
vertices, as desired. 
\\
On the $V_n$ network, for $n\geq0$, the 3 corners born on level 0 have degree
\begin{equation*}
 2+\sum_{j=1}^{n}2^j=2^{n+1},
\end{equation*}
which verifies item (i).
\\
Following the construction, we see that on the $V_n$ network, for $n\geq1$, there are $6^{n-1}$ centers born at level $n$, each with degree 12. There are $6^{n-2}$ centers born at level $n-1$, each with degree $12+12$. In general, for $0\leq k\leq n$, there are $6^{n-k-1}$ centers born at level n-k, each with degree
\begin{equation*}
 12 + 12\cdot \sum_{j=0}^{k-1} 2^j = 3\cdot2^{k+2}.
\end{equation*}
After changing indices, item (ii) follows, noting that item (ii) is a vacuous statement for $n=0$. 
\\
Similarly, for $0\leq k\leq n$, in the $V_n$ network, there are $3\cdot6^{n-k-1}$ corners born at level $n-k$. Each of which have degree
\begin{equation*}
 4 + 4\cdot \sum_{j=0}^{k-1}2^j = 2^{k+2}.
\end{equation*}
After changing indices, item (iii) follows, noting that item (iii) is a vacuous statement for $n=0$. 
\end{proof}

\begin{corollary}\label{cor:vertexdegree1} For the $V_n$ network of the non-p.c.f. analog of the Sierpi\'{n}ski gasket, for $n\geq 1$, we have
\begin{equation*}
\frac{\prodd{j}{|V_n|}}{\sumd{j}{|V_n|}}=2^{\frac{1}{25}\left(44\cdot 6^n+30n+6\right)}\cdot 3^{\frac{1}{5}\left(6^n-5n-6\right)}.
\end{equation*}

\end{corollary}
%\begin{proof}[Proof of Corollary~\ref{cor:vertexdegree1}]  From Lemma~\ref{thm:vertexdegrees1}, we know that
%\begin{align*}\prodd{j}{|V_n|}&=\left(2^{n+1}\right)^3\cdot\prod_{k=1}^n\left(3\cdot 2^{n-k+2}\right)^{6^{k-1}}\cdot\prod_{k=1}^{n}\left(2^{n-k+2}\right)^{3\cdot 6^{k-1}}\\
%&=2^{3n+3}\cdot 3^{\sum_{k=1}^n 6^{k-1}}\cdot 2^{4\cdot \sum_{k=1}^n (n-k+2)\cdot 6^{k-1}}\\
%&=2^{\frac{1}{25}\left(44\cdot 6^n+55n+31\right)}\cdot 3^{\frac{1}{5}\cdot \left(6^n-1\right)},
%\end{align*}
%and that
%\begin{align*}
%\sumd{j}{|V_n|}&=3\cdot 2^{n+1}+\sum_{k=1}^n 6^{k-1}\left(3\cdot 2^{n-k+2}\right)+3\cdot\sum_{k=1}^n 6^{k-1}\left(2^{n-k+2}\right)\\
%&=3\cdot 2^{n+1}+\sum_{k=1}^n 6^k\cdot 2^{n-k+2}\\
%&=3\cdot 2^{n+1}+2^{n+2}\sum_{k=1}^n 3^k\\
%&=3\cdot 2^{n+1}\left(1+2\sum_{k=1}^n 3^{k-1}\right)\\
%&=2^{n+1}\cdot 3^{n+1}.
%\end{align*}
%Combining these calculations, the Corollary follows.
%\end{proof}
We are now ready for the proof of the main theorem in this section.
\begin{proof}[Proof of Theorem~\ref{thm:nonpcf}]  We apply Theorem~\ref{thm:maintheoremfull}.  In \cite{MR2451619}, they use a result from \cite{MR2450694} to carry out spectral decimation for the non-p.c.f. analog of the Sierpi\'{n}ski gasket.  In our language, they showed that 
\begin{align*}
A&=\left\{\frac{3}{2}\right\},\text{ and }
B=\left\{\frac{3}{4},\frac{5}{4},\frac{1}{2},1\right\}.
\end{align*}
Rephrasing their results in our language, for $n\geq 2$ the following hold:
\begin{enumerate}[(I)]
\item\label{first} $\alpha=\frac{3}{2}$,\;\; $\alpha_n=6^{n-1}+1$,
\item\label{second} $\beta=\frac{3}{4}$,
\begin{equation*}\beta_n^k=\begin{cases} 6^{n-k-2}+1&\;\;\;\;\;k=0,\ldots,n-2\\
2&\;\;\;\;\;k=n-1\\
0&\;\;\;\;\;k=n,\\
\end{cases}
\end{equation*}
\item\label{third} $\beta=\frac{5}{4}$,
\begin{equation*}\beta_n^k=\begin{cases} 6^{n-k-2}+1&\;\;\;\;\;k=0,\ldots,n-2\\
2&\;\;\;\;\;k=n-1\\
0&\;\;\;\;\;k=n,\\
\end{cases}
\end{equation*}
\item\label{fourth} $\beta=\frac{1}{2}$,
\begin{equation*}\beta_n^k=\begin{cases} \frac{\displaystyle 11\cdot 6^{n-k-2}-6}{\displaystyle 5}&\;\;\;\;\;k=0,\ldots,n-2\\
0&\;\;\;\;\;k=n-1,n,\\
\end{cases}
\end{equation*}
\item\label{fifth} $\beta=1$,
\begin{equation*}\beta_n^k=\begin{cases} \frac{\displaystyle 6^{n-k}-6}{\displaystyle 5}&\;\;\;\;\;k=0,\ldots,n-2\\
0&\;\;\;\;\;k=n-1,n,\\
\end{cases}
\end{equation*}
\end{enumerate}
and
\begin{equation*}
R(z)=\frac{\displaystyle -24z(z-1)(2z-3)}{\displaystyle 14z-15}.
\end{equation*}
So $d=3$, $Q(0)=-15$ and $P_d=-48$.

We now use Equation~\ref{eqn:maintheoremfull} in Theorem~\ref{thm:maintheoremfull} to calculate $\tau(V_n)$.  We have from (\ref{first}),
\begin{equation}
\prod_{\alpha\in A} \alpha^{\alpha_n}=\left(\frac{3}{2}\right)^{6^{n-1}+1}.
\end{equation}
From (\ref{second}),(\ref{third}),(\ref{fourth}), and (\ref{fifth}), and to calculate
\begin{equation}
\begin{split}
\prod_{\beta\in B}&\left(\beta^{\sum_{k=0}^n\beta_n^k}\cdot\left(\frac{15}{48}\right)^{\sum_{k=0}^n\beta_n^k\left(\frac{d^k-1}{d-1}\right)}\right),\\
%&=\left(\frac{3}{4}\right)^{\displaystyle \left[\sum_{k=0}^{n-2}\left(6^{n-k-2}+1\right)\right]+2}\times\left(\frac{15}{48}\right)^{\displaystyle \left[\sum_{k=0}^{n-2}\left(6^{n-k-2}+1\right)\left(\frac{3^k-1}{2}\right)\right]+2\cdot\frac{3^{n-1}-1}{2}}\\
%&\times\left(\frac{5}{4}\right)^{\displaystyle \left[\sum_{k=0}^{n-2}\left(6^{n-k-2}+1\right)\right]+2}\times\left(\frac{15}{48}\right)^{\displaystyle \left[\sum_{k=0}^{n-2}\left(6^{n-k-2}+1\right)\left(\frac{3^k-1}{2}\right)\right]+2\cdot\frac{3^{n-1}-1}{2}}\\
%&\times\left(\frac{1}{2}\right)^{\displaystyle \sum_{k=0}^{n-2}\frac{11\cdot 6^{n-k-2}-6}{5}}\times\left(\frac{15}{48}\right)^{\displaystyle \sum_{k=0}^{n-2}\left(\frac{11\cdot 6^{n-k-2}-6}{5}\right)\left(\frac{3^k-1}{2}\right)}\\
%&\times\biggl(1\biggr)^{\displaystyle \sum_{k=0}^{n-2}\frac{6^{n-k}-6}{5}}\times\left(\frac{15}{48}\right)^{\displaystyle \sum_{k=0}^{n-2}\left(\frac{6^{n-k}-6}{5}\right)\left(\frac{3^k-1}{2}\right)}
\end{split}
\end{equation}
the relevant summations are, 

\begin{align*}
\left[\sum_{k=0}^{n-2}\left(6^{n-k-2}+1\right)\right]+2&=\frac{1}{5}\left(6^{n-1}+5n+4\right),\\
\left[\sum_{k=0}^{n-2}\left(6^{n-k-2}+1\right)\left(\frac{3^k-1}{2}\right)\right]+\left(3^{n-1}-1\right)&=\frac{1}{60}\left(4\cdot 6^{n-1}+65\cdot 3^{n-1}-30n-39\right),\\
\sum_{k=0}^{n-2}\frac{11\cdot 6^{n-k-2}-6}{5}&=\frac{1}{25}\left(11\cdot 6^{n-1}-30n+19\right),\\
\sum_{k=0}^{n-2}\left(\frac{11\cdot 6^{n-k-2}-6}{5}\right)\left(\frac{3^k-1}{2}\right)&=\frac{1}{25}\left(22\cdot 6^{n-2}-50\cdot 3^{n-2}+15n-2\right),\text{ and}\\
\sum_{k=0}^{n-2}\left(\frac{6^{n-k}-6}{5}\right)\left(\frac{3^k-1}{2}\right)&=\frac{1}{50}\left(4\cdot 6^n-25\cdot 3^n+30n+21\right).
\end{align*}

All of these equations are valid for $n\geq 2$ and combining with Corollary~\ref{cor:vertexdegree1}, we see that 
$$\tau(V_n)=2^{f_n}\cdot 3^{g_n}\cdot 5^{h_n},\;\;\; n\geq 2$$
where $f_n$, $g_n$, and $h_n$ are as claimed.
%\begin{align*}
%f_n&=\frac{2}{25}\left(11\cdot 6^n-30n-11\right),\\
%g_n&=\frac{1}{5}\left(2\cdot 6^n+3\right),\text{ and}\\
%h_n&=\frac{1}{25}\left(4\cdot 6^n+30n-4\right).
%\end{align*}
For $n=0$, since the $V_0$ graph is the complete graph on three vertices, $\tau(V_0)=3$ by Cayley's Formula, as desired. For $n=1$, from  \cite{MR2451619} the eigenvalues of $P_1$ are $\{\frac{5}{4},\frac{5}{4},\frac{3}{2},\frac{3}{2},\frac{3}{4},\frac{3}{4},0\}$ and using Corollary ~\ref{cor:vertexdegree1} for $n=1$, we apply Theorem~\ref{thm:matrixtree}
to see that $\tau(V_1)=2^2\cdot3^3\cdot5^2$, as desired.
\end{proof}
\begin{corollary} The asymptotic growth constant for the non-p.c.f. analog of the $\sierp$ Gasket is 
\begin{equation}
c=\frac{11\cdot log(2)}{10}+\frac{log(3)}{2}+\frac{log(5)}{5}
 \end{equation}
\end{corollary}

%\begin{proof}
%Use Theorem \ref{thm:nonpcf} and recall that 
%$$|V_n|=\frac{4\cdot 6^n+11}{5}$$

%\end{proof}

\subsection{Diamond Fractal}\label{section:diamond}
The diamond self-similar hierarchical lattice appeared as an example in several physics works, including \cite{MR706699}, \cite{MR734134}, and ~\cite{MR755653}.  
In \cite{MR2450694}, the authors modify the standard results for the unit interval $[0,1]$ to develop the spectral decimation method for this fractal, hence Theorem \ref{thm:maintheoremfull} still applies.
Figure~\ref{fig:diamond} shows the $V_1$ and $V_2$ networks for this.

\begin{figure}[h!]
\begin{center}
\epsfig{file=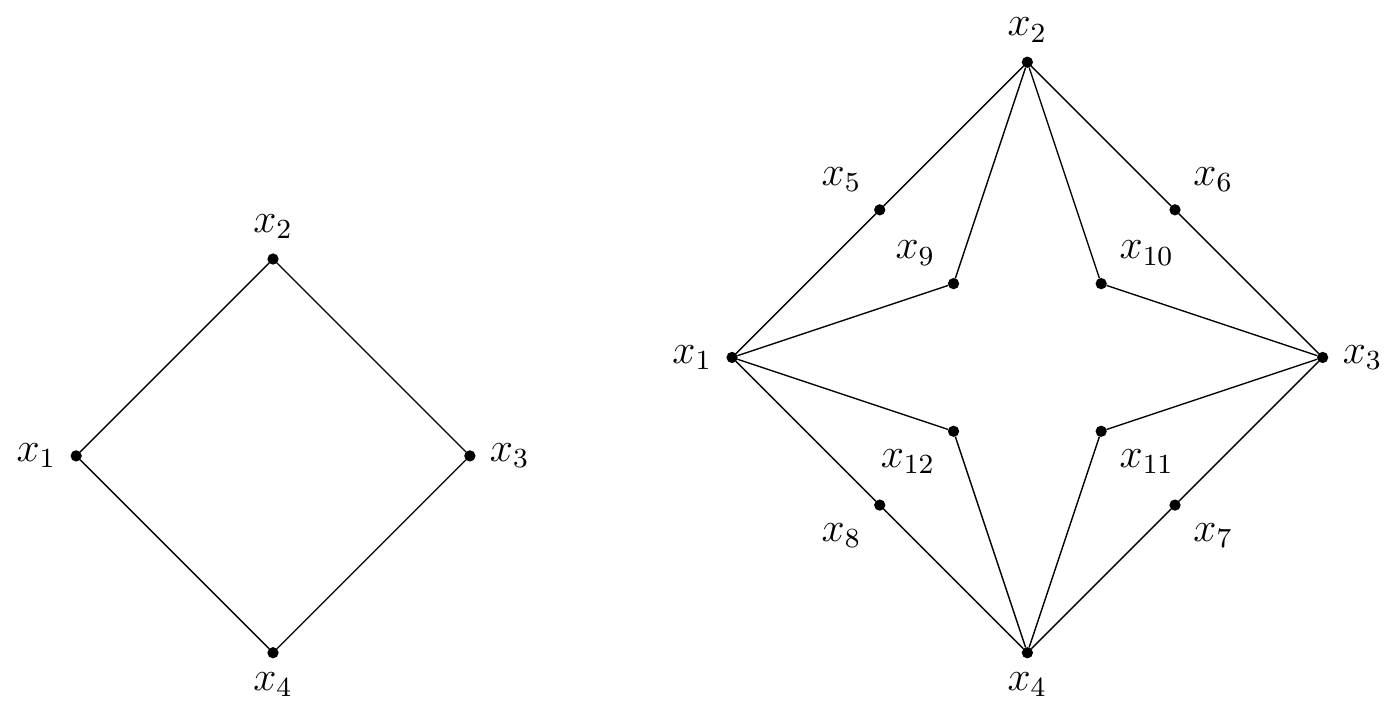,scale=.5}

\caption{The $V_1$ and $V_2$ network of the Diamond fractal.}
\end{center}
\label{fig:diamond}
\end{figure}

%\begin{figure}[h!]
%\centering
%\includegraphics[scale=.45]{V2diamond.pdf}
%\caption{taining the other three, and still form an embedded $K_4$.}
%\label{fig:diamond}
%\end{figure}

\begin{theorem}\label{thm:spandiamond} The number of spanning trees on the Diamond fractal at level $n$ is given by
$$\tau(V_n)=2^{\frac{2}{3}\left(4^{n}-1\right)}\;\;\;\;\;n\geq 1.$$
\end{theorem}
Before we begin the proof, we need a few results.
\begin{lemma}\label{thm:diamond}  The $V_n$ network of the Diamond fractal, for $n\geq 1$, has
$$\frac{\left(4+2\cdot 4^n\right)}{3}$$
vertices.  Among these vertices,
\begin{enumerate}[(i)]
\item $2\cdot 4^{n-k}$ have degree $2^k$ for $1\leq k\leq n-1$
\item $4$ have degree $2^n$.
\end{enumerate}
\end{lemma}
\begin{remark}  In \cite{MR2450694}, the number of vertices of $V_n$ is incorrect as stated in Theorem 7.1(ii).  We correct this here and provide a proof.
\end{remark}
\begin{proof}[Proof of Lemma~\ref{thm:diamond}] 
We first describe how the $V_n$ network is constructed, then prove the Lemma. 
 When $n=1$, $V_1$ has four vertices of degree 2 and one diamond, this diamond is the graph of $V_1$. We say these vertices and diamond are born at level 1. 

When $n=2$, from the diamond born on level 1, 4 diamonds are born. We say these diamonds are born on level 2. For each of the diamonds born on level 2, 2 vertices of degree 2 are born. We say these vertices are born on level 2. Using the notation $G=<V,E>$ where $G$ is the graph, $V$ is the graph's vertex set and $E$ is the graph's edge set. An example diamond born at level 2 is $<V,E>$, where 
\begin{equation*}
V=\{x_1,x_5,x_2,x_9\}
\end{equation*}
\begin{equation*}
E=\{x_1x_5,x_5x_2,x_2x_9,x_9x_1\}
\end{equation*}
 which gives birth to $x_5$ and $x_9$. Every vertex born on level 1 gains 2 more edges.

For $n\geq2$, from each diamond born on level $n-1$, 4 diamonds are born at level n. For each of the diamonds born on level n, 2 vertices of degree 2 are born at level n. Every vertex born on level $n-1$, gains 2 more edges. Every vertex born on level $n-2$, gains $2^2$ more edges. In general, every vertex born on level $n-k$, gains $2^k$ more edges for $1\leq k\leq n-1$.

From this construction, we see that at level n, for $n\geq 1$, there are $4^{k-1}$ diamonds born at level k, $1\leq k\leq n$, $2\cdot 4^{k-1}$ vertices born at level k, $2\leq k\leq n$ and 4 vertices born at level 1. Thus, the $V_n$ network has 
\begin{align*}
 4+ \sum_{k=2}^{n}2\cdot 4^{k-1}
&=\frac{(4+2\cdot4^n)}{3}  \textrm{ vertices, as desired.}
\end{align*}
In the $V_n$ network, the 4 vertices born at level 1 have degree
\begin{align*}
 2+\sum_{j=1}^{n-1}2^j = 2^n,
\end{align*}
which verifies item (ii) of the Proposition. \\
In the $V_n$ network, the $2\cdot 4^{k-1}$ vertices born on level k, $2\leq k\leq n$, have degree
\begin{align*}
 2+\sum_{j=1}^{n-k}2^j = 2^{n-k+1}. 
\end{align*}
changing indices, this verifies item (i) of the Lemma. 
\end{proof}

\begin{corollary}\label{cor:diamond}  For the $V_n$ network of the Diamond fractal, we have
\begin{equation}\label{eqn:pidiamond}
\frac{\displaystyle \prodd{i}{|V_n|}}{\displaystyle \sumd{i}{|V_n|}}=2^{\frac{1}{9}\left(2\cdot 4^{n+1}-6n-17\right)}.
\end{equation}
\end{corollary}
%\begin{proof}[Proof of Corollary~\ref{cor:diamond}]  From Lemma~\ref{thm:diamond}, we know that
%\begin{align*}
%\prodd{i}{|V_n|}&=\left(2^{n}\right)^4\cdot \prod_{k=1}^{n-1}\left(2^k\right)^{2\cdot 4^{n-k}}\\
%&=2^{4n}\cdot 2^{\sum_{k=1}^{n-1}2\cdot k\cdot 4^{n-k}}\\
%&=2^{\frac{1}{9}\left(2\cdot 4^{n+1}+12n-8\right)}
%\end{align*}
%Similarly, we have
%\begin{align*}
%\sumd{i}{|V_n|}&=\left(\sum_{k=1}^{n-1}2\cdot 2^k\cdot 4^{n-k}\right)+4\cdot 2^n\\
%&=2^{2n+1}.
%\end{align*}
%Combining these calculations, the corollary follows.
%\end{proof}
We now return to a proof the the main theorem of this section.
\begin{proof}[Proof of Theorem~\ref{thm:spandiamond}]  We apply Theorem~\ref{thm:maintheoremfull}.  In \cite{MR2450694}, they carry out spectral decimation for the Diamond fractal.  In our language, they showed that
\begin{align*}
A&=\left\{2\right\}, \text{ and }
B=\left\{1\right\}.\\
\end{align*}
For $n\geq 1$, the following hold:
\begin{enumerate}[(I)]
\item $\alpha=2$, $\alpha_n=1$
\item $\beta=1$, \begin{equation*}\beta_n^k=\begin{cases} \frac{4^{n-k}+2}{3}&\;\;\;\;\;k=0,\ldots,n-1\\
0&\;\;\;\;\;k=n,\\
\end{cases}
\end{equation*}
\end{enumerate}
and
$$R(z)=2z(2-z).$$
So $d=2$, $Q(0)=1$, and $P_d=-2$.
We now use Equation~\ref{eqn:maintheoremfull} in Theorem~~\ref{thm:maintheoremfull} to calculate $\tau(V_n)$.
\begin{equation}\label{eqn:diamondalpha}
\prod_{\alpha\in A} \alpha^{\alpha_n}=2^1
\end{equation}
\begin{equation}\label{eqn:diamondbeta}
\begin{split}
\prod_{\beta\in B}&\left(\beta^{\sum_{k=0}^n\beta_n^k}\cdot\left(\frac{1}{2}\right)^{\sum_{k=0}^n\beta_n^k\left(2^k-1\right)}\right)
=2^{-\frac{1}{9}\left(2\cdot 4^n-6n-2\right)}
%&=1^{\displaystyle \sum_{k=0}^{n-1}\left(\frac{4^{n-k}+2}{3}\right)}\times\left(\frac{1}{2}\right)^{\displaystyle \sum_{k=0}^{n-1}\left(\frac{4^{n-k}+2}{3}\right)\left(2^k-1\right)}\\
\end{split}
\end{equation}
the relevant summation is,
\begin{equation*}
\sum_{k=0}^{n-1}\left(\frac{4^{n-k}+2}{3}\right)\left(2^k-1\right)=\frac{1}{9}\left(2\cdot 4^n-6n-2\right).
\end{equation*}
Combining this with Corollary~\ref{cor:diamond}, we have that
$$\tau(V_n)=2^{\frac{2}{3}\left(4^{n}-1\right)}\;\;\;\;\;n\geq 1$$
as desired.
\end{proof}

\begin{corollary} The asymptotic growth constant for the Diamond fractal is 
\begin{equation}
c=log(2)
 \end{equation}
\end{corollary}

%\begin{proof}
%Use Theorem \ref{thm:spandiamond} and recall that 
%$$|V_n|=\frac{\left(4+2\cdot 4^n\right)}{3}$$

%\end{proof}

\subsection{Hexagasket}\label{section:hexa}

The hexagasket, is also known as the hexakun, a polygasket, a 6-gasket, or a $(2,2,2)$-gasket, see \cite{MR2451619,Ki01, ASST03, BCF07, St06, Te08, TW05, TW06}. The $V_1$ network of the hexagasket is shown in the figure below. 
\begin{figure}[h!]
\begin{center}
\epsfig{file=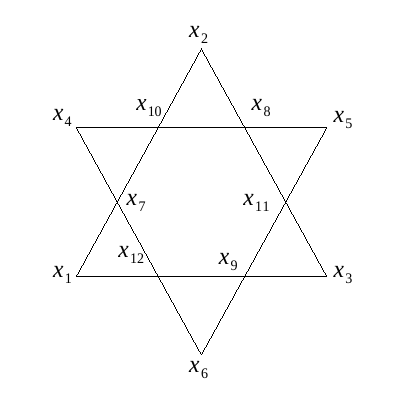,scale=1}
\\
\caption{The $V_1$ network of the Hexagasket.}
\label{fig:hexav1}
\end{center}
\end{figure}

\begin{theorem}\label{thm:hexa} The number of spanning trees on the Hexagasket at level $n$ is given by
$$\tau(V_n)=2^{f_n}\cdot3^{g_n}\cdot7^{h_n}\;\;\;\;\;n\geq 0.$$
where
\begin{align*}
f_n&=\frac{1}{225}\left(27\cdot6^{n+1}-100\cdot4^n-60n-62\right)\\
g_n&=\frac{1}{25}\left(4\cdot6^{n+1}+5n+1\right)\\
h_n&=\frac{1}{25}\left(6^n-5n-1\right).\\
\end{align*}
\end{theorem}
\begin{proof}[Proof of Theorem~\ref{thm:hexa}]
We apply Theorem~\ref{thm:maintheoremfull}. From \cite{MR2451619} it is known that 
\begin{equation*}
 |V_n|=\frac{(6+9\cdot6^n)}{5} \;\;\;\;\; n\geq0,
\end{equation*}
of these vertices, 
%\begin{equation*}
$ \frac{6(6^n-1)}{5} \textrm{ have degree } 4,$
%\end{equation*}
 and the remaining
%\begin{equation*}
 $\frac{(12 + 3\cdot6^n)}{5} \textrm{ have degree } 2.$
%\end{equation*}
So we compute
\begin{equation}\label{eqn:hexdegree}
\frac{\prodd{j}{|V_n|}}{\sumd{j}{|V_n|}}=2^{(3\cdot6^n-n-1)}\cdot 3^{-(n+1)}
\end{equation}
for $n\geq 0$.
\\
In \cite{MR2451619}, they use a result from \cite{MR2450694} to carry out spectral decimation for the Hexagasket.
We note that in \cite{MR2451619} Theorem 6.1 (v) and (vi), the bounds on $k$ should be $0\leq k\leq n-1$ and in (vii) the bounds should be $0\leq k\leq n-2$. This can be verified using Table 2 in the same paper. In our language they showed that 
\begin{align*}
A&=\left\{\frac{3}{2}\right\},\text{ and }
B=\left\{1,\frac{1}{4},\frac{3}{4},\frac{3+\sqrt{2}}{4},\frac{3-\sqrt{2}}{4}\right\},
\end{align*}
and for $n\geq 2$ the following hold:
\begin{enumerate}[(I)]
\item\label{first} $\alpha=\frac{3}{2}$,  \;\;\;$\alpha_n=\frac{(6+4\cdot6^n)}{5}$,
\item\label{second} $\beta=1$,
\begin{equation*}\beta_n^k=\begin{cases} 1&\;\;\;\;\;k=0,\ldots,n-1\\
0&\;\;\;\;\;k=n,\\
\end{cases}
\end{equation*}
\item\label{third} $\beta=\frac{1}{4},\frac{3}{4}$,
\begin{equation*}\beta_n^k=\begin{cases} \frac{(6+4\cdot6^{n-k-1})}{5}&\;\;\;\;\;k=0,\ldots,n-1\\
0&\;\;\;\;\;k=n,\\
\end{cases}
\end{equation*}
\item\label{fourth} $\beta=\frac{3+\sqrt{2}}{4},\frac{3-\sqrt{2}}{4}$,
\begin{equation*}\beta_n^k=\begin{cases} \frac{(6^{n-k-1}-1)}{5}&\;\;\;\;\;k=0,\ldots,n-2\\
0&\;\;\;\;\;k=n-1,n,\\
\end{cases}
\end{equation*}

\end{enumerate}

\begin{equation*}
R(z)=\frac{2z(z-1)(7-24z+16z^2)}{(2z-1)}.
\end{equation*}
So $d=4$, $Q(0)=-1$ and $P_d=32$.
\\
We now use equation~\ref{eqn:maintheoremfull} in Theorem~\ref{thm:maintheoremfull} to calculate $\tau(V_n)$. 
 The relevant sums are
\begin{equation}
 \sum_{k=0}^{n-1}\frac{(4^k-1)}{3}= \frac{(4^n-3n-1)}{9}
\end{equation}
\begin{equation}
 \sum_{k=0}^{n-1}\frac{(6+4\cdot6^{n-k-1})}{5}=\frac{2\cdot(2\cdot6^n +15n-2)}{25}
\end{equation}
\begin{equation}
 \sum_{k=0}^{n-1}\frac{(6+4\cdot6^{n-k-1})}{5}\frac{(4^k-1)}{3}=\frac{(6^{n+1}-30n-6)}{75}
\end{equation}
\begin{equation}
 \sum_{k=0}^{n-2}\frac{(6^{n-k-1}-1)}{5}=\frac{(6^n-5n-1)}{25}
\end{equation}
\begin{equation}
  \sum_{k=0}^{n-2}\frac{(6^{n-k-1}-1)}{5}\frac{(4^k-1)}{3}=\frac{(9\cdot6^n-25\cdot4^n+30n+16)}{450}
\end{equation}
Combining these using equations ~\ref{eqn:maintheoremfull} and ~\ref{eqn:hexdegree}, after simplifying we get 

$$\tau(V_n)=2^{f_n}\cdot3^{g_n}\cdot7^{h_n}\;\;\;\;\;n\geq 2.$$
Where $f_n,g_n,$ and $h_n$ are as claimed. 
\\
For n=1, equation \ref{eqn:hexdegree} still holds and from \cite{MR2451619} we know the eigenvalues of the probabilistic graph Laplacian on $V_1$ are \{$1,\frac{1}{4} ,\frac{1}{4},\frac{3}{4},\frac{3}{4},\frac{3}{2},\frac{3}{2},\frac{3}{2},\frac{3}{2},\frac{3}{2},\frac{3}{2},0$\}.
So by Theorem~\ref{thm:matrixtree}, we get that $\tau(V_1)=2^2\cdot 3^6$, thus the theorem holds for $n=1$.  The $V_0$ network is the complete graph on 3 vertices, thus $\tau(V_0)=3$. Hence the theorem holds for all $n\geq 0$.
\end{proof}
\begin{corollary} The asymptotic growth constant for the Hexagasket is 
\begin{equation}
c=\frac{2\cdot log(2)}{5}+\frac{8\cdot log(3)}{15}+\frac{log(7)}{45}
 \end{equation}
\end{corollary}

%\begin{proof}
%Use Theorem \ref{thm:hexa} and recall that 
%$$|V_n|=\frac{(6+9\cdot6^n)}{5}$$

%\end{proof}

%\nocite*
\bibliography{AnemaThesis_newbib}
\end{document}